\documentclass[12pt]{amsart}
\usepackage{amssymb,amsmath,amsthm,latexsym}
\usepackage{mathrsfs}
\usepackage{mdwlist}
\usepackage{graphicx}
\usepackage{a4wide}
\usepackage{chessfss}

\newcommand{\R}{\mathbb{R}}

\newcommand{\N}{\mathbb{N}}

\newcommand{\C}{\mathbb{C}}

\newcommand{\p}{\varphi}
\newcommand{\e}{\varepsilon}

\newcommand{\oo}{\overline}
\newcommand{\om}{\omega}

\newcommand{\ind}{\boldsymbol{1}}
\newcommand{\n}[1]{\|#1\|}

\newcommand{\supp}{\mathrm{supp}}
\newcommand{\ccup}{\scalebox{0.85}{$\bigcup$}}

\newcommand{\Ca}{\mathscr{C}}

\newtheorem{theorem}{Theorem}[section]
\newtheorem{lemma}[theorem]{Lemma}

\newtheorem{proposition}[theorem]{Proposition}
\newtheorem{corollary}[theorem]{Corollary}

\theoremstyle{definition}

\newtheorem{example}[theorem]{Example}
\newtheorem{remark}[theorem]{Remark}
\newtheorem{question}[theorem]{Question}

\theoremstyle{remark}

\numberwithin{equation}{section}

\newcommand{\abs}[1]{\lvert#1\rvert}

\newcounter{smallromans}

\newenvironment{romanenumerate}
{\begin{list}{{\normalfont\textrm{(\roman{smallromans})}}}%
    {\usecounter{smallromans}\setlength{\itemindent}{0cm}%
      \setlength{\leftmargin}{5.5ex}\setlength{\labelwidth}{5.5ex}%
      \setlength{\topsep}{0.75\parsep}\setlength{\partopsep}{0ex}%
      \setlength{\itemsep}{0ex}}}%
  {\end{list}}

\newcounter{smallalphs}

  {\end{list}}

\begin{document}
\title{A chain condition for operators from $C(K)$-spaces}
\author[Hart, Kania and Kochanek]{Klaas Pieter Hart, Tomasz Kania and Tomasz Kochanek}
\address{Faculty of Electrical Engineering, Mathematics and Computer Science\\
         TU Delft\\
         Postbus 5031\\
         2600~GA {} Delft\\
         the Netherlands}
\email{k.p.hart@tudelft.nl}

\address{Department of Mathematics and Statistics, Fylde College,
  Lancaster University, Lancaster LA1 4YF, United Kingdom}
\email{t.kania@lancaster.ac.uk}

\address{Institute of Mathematics, University of Silesia, Bankowa 14, 40-007 Katowice, Poland}
\email{tkochanek@math.us.edu.pl}


\begin{abstract}
We introduce a chain condition ($\symbishop$), defined for operators
acting on $C(K)$-spaces, which is intermediate between weak
compactness and having weakly compactly generated range. It is
motivated by Pe\l czy\'nski's characterisation of weakly compact
operators on $C(K)$-spaces. We prove that if $K$ is extremally
disconnected and $X$ is a~Banach space then, for an~operator
$T\colon C(K)\to X$, $T$ is weakly compact if and only if $T$
satisfies ($\symbishop$) if and only if the representing vector
measure of $T$ satisfies an~analogous chain condition. As a~tool for proving the above-mentioned result, we derive a~topological counterpart of Rosenthal's lemma. We exhibit
several compact Hausdorff spaces $K$ for which the identity operator
on $C(K)$ satisfies ($\symbishop$), for example both locally
connected compact spaces having countable cellularity~and ladder
system spaces have this property. Using a~Ramsey-type theorem,
due to Dushnik and Miller, we prove that the collection of operators
on a~$C(K)$-space satisfying {\rm($\symbishop$)} forms a~closed left
ideal of $\mathscr{B}(C(K))$.\end{abstract}

\maketitle

\section{Introduction}
The aim of this paper is to study a~certain chain condition, denoted {\rm($\symbishop$)}, defined for (bounded, linear) operators $T\colon C(K)\to X$, where $K$ is a~compact Hausdorff space, $X$ is an~arbitrary Banach space and $C(K)$ is the Banach space of all scalar-valued continuous functions on $K$, equipped with the supremum norm. The main motivation behind our work comes from a~theorem of Pe\l czy\'{n}ski which asserts that an~operator $T\colon C(K)\to X$ is weakly compact if and only if there is no isomorphic copy of $c_0$ in $C(K)$ on which $T$ is bounded below. In other words, $T$ fails to be weakly compact provided there is a~sequence of pairwise disjoint open sets $\{O_n\}_{n=1}^\infty$ in $K$ and a~uniformly bounded family of functions $\{f_n\}_{n=1}^\infty$ in $C(K)$ such that the support of $f_n$ is contained in $O_n$ ($n\in \N$) and $\inf_n \|T(f_n)\|>0$. Loosely speaking, Pe\l czy\'nski's theorem characterises weakly compact operators from $C(K)$-spaces as precisely those which annihilate the sequences of disjointly supported and uniformly bounded functions in $C(K)$. Our condition ($\symbishop$) is similar in nature to this requirement.

Let $K$ be a compact Hausdorff space and $X$ be a~Banach space. The \emph{support}, $\supp(f)$, of a~function $f\in C(K)$ is defined as
the closure of the set $\{x\in K\colon f(x)\neq 0\}$. For any~pair
of functions $f,g\in C(K)$ we write $f\prec g$ whenever $f\not=g$
and $f(x)=g(x)$ for each $x\in\supp(f)$. The relation $\prec$ is
a~strict ordering on $C(K)$. For a~bounded and linear operator
$T\colon C(K)\to X$, the chain condition ($\symbishop$) is defined
as follows:
\begin{itemize*}
\item[($\symbishop$)]
for each uncountable $\prec$-chain $F$ in $C(K)$ with $\sup_{f\in F}\n{f}<\infty$,
$$\inf\{\n{Tf - Tg}\colon f,g\in F, f\neq g\}=0.$$
\end{itemize*}
Any uniformly bounded $\prec$-chain $\{f_i\}_{i\in I}$ such that
$\n{f_i-f_j}\geqslant\delta$ for $i,j\in I$, $i\neq j$, and some
positive $\delta$, will be called a~$\delta$-$\prec$-{\it chain}. We
also say that a~compact Hausdorff space $K$ satisfies ($\symbishop$)
whenever the identity operator on $C(K)$ satisfies ($\symbishop$),
equivalently, for each $\delta>0$ there is no uncountable
$\delta$-$\prec$-chain in $C(K)$.

We start Section 2 with the observations that, for each operator
$T\colon C(K)\to X$, the condition {\rm($\symbishop$)} is weaker
than weak compactness (Proposition \ref{wcompactisbishop}), and
that these two properties coincide provided that $K$ is extremally
disconnected (i.e. open sets have open closures), in which case weak
compactness may also be characterised by a~counterpart of 
condition ($\symbishop$) for the representing measure of $T$
(Theorem \ref{bishop_rook}). This result is obtained by refining
a~classical lemma due to Rosenthal to the extremally disconnected
setting (Lemma \ref{rosenthal}) and by applying some vector-measure
techniques (Proposition \ref{measure_1}).

In Section 3 we give examples of compact Hausdorff spaces that
satisfy {\rm($\symbishop$)}. They include compact Hausdorf spaces
which are countable, one-point compactifications of discrete sets
(Proposition \ref{thm:omegaD}),
 compact Hausdorff spaces which are locally connected and have
 countable cellularity (Theorem \ref{souslin}) and ladder system
spaces (Corollary \ref{ladder}).

Like other classical chain conditions for topological spaces or
Boolean algebras, this one carries some combinatorial insight. Using
Ramsey-theoretic techniques we prove that the family of operators on
a given $C(K)$-space satisfying {\rm($\symbishop$)} forms a closed
left ideal of the Banach algebra $\mathscr{B}(C(K))$ of all
operators on $C(K)$ (Theorem \ref{left_ideal}).

All Banach spaces are assumed to be over either real or complex
scalars. By an \emph{operator} we understand a bounded linear
operator acting between Banach spaces. An operator $T\colon E\to F$
is \emph{bounded below} whenever there exists $\gamma>0$ such that
$\|Tx\|\geqslant \gamma\|x\|$ for each $x\in E$; an operator which
is bounded below is injective and has closed range. An operator
$T\colon E\to F$ \emph{fixes a copy} of a Banach space $X$, if
there is a subspace $E_0$ of $E$ isomorphic to $X$ such that
$T|_{E_0}$ is bounded below. The identity operator on a~Banach space
$X$ is denoted $I_X$. For a topological space $K$, we denote its
cardinal number by $|K|$, and define the \emph{cellularity} of $K$, $c(K)$,
to be the supremum of the cardinalities of families consisting of
pairwise disjoint open sets in $K$. A topological space with
countable cellularity is said to satisfy \emph{c.c.c}. The
symbol $\mathbf{1}_A$ denotes the indicator function of a set $A$.

\section{{\rm($\symbishop$)} and weakly compact operators from $C(K)$-spaces}
By a classical result of Pe\l czy\'nski (see, \emph{e.g.}, \cite[Theorem VI.2.15]{diestel_uhl}), every non-weakly compact operator from a~$C(K)$-space into an~arbitrary Banach space $X$ fixes a copy of $c_0$. In the case where $K$ is extremally disconnected the Goodner--Nachbin theorem (\cite{goodner}, \cite{nachbin}) says that $C(K)$ is (isometrically) injective, and then Rosenthal's theorem \cite[Theorem 1.3]{rosenthal} asserts that any non-weakly compact operator $T\colon C(K)\to X$ fixes a copy of $\ell_\infty$. The main result of this section, Theorem \ref{bishop_rook}, is in the spirit of these facts. It asserts that weak compactness of an operator from a~$C(K)$-space, for $K$ extremally disconnected, may be characterised in terms of the chain conditions, imposed both upon the operator itself and its representing measure.

Let us begin with the following refinement of Pe\l czy\'nski's theorem.
\begin{proposition}\label{wcompactisbishop}Let $K$ be a compact Hausdorff space, $X$ be a Banach space and let $T\colon C(K)\to X$ be an operator. Then the following assertions are equivalent
\begin{romanenumerate}
\item\label{wcompactpishopi} $T$ is weakly compact;
\item\label{wcompactpishopii} for each infinite $\prec$-chain $F\subseteq C(K)$, $$\inf\{\|Tf-Tg\|\colon f,g\in F, f\neq g\}=0.$$
\end{romanenumerate}
In particular, every weakly compact operator $T\colon C(K)\to X$ satisfies ${\rm (\symbishop)}$.\end{proposition}
\begin{proof}
For the implication (\ref{wcompactpishopi}) $\Rightarrow$
(\ref{wcompactpishopii}) assume contrapositively that
$$\delta=\inf\{\|Tf-Tg\|\colon f,g\in F, f\neq g\}>0.$$ Let $F_0=(f_n)_{n\in \mathbb{N}}$ be a $\prec$-monotone subsequence of $F$. For each $n\in \mathbb N$ set $g_n = |f_{n+1}-f_n|$. Then $Y=\overline{\mbox{span}}\{g_n\colon n\in \mathbb{N}\}$ is an isomorphic copy of $c_0$ and $\|Tg_n\|>\delta$ ($n\in \mathbb{N}$). Consequently, $T$ is bounded below on $Y$. So $T$ cannot be weakly compact.

To prove (\ref{wcompactpishopii}) $\Rightarrow$
(\ref{wcompactpishopi})
 assume, for a contradiction, that
  $T$ is not weakly compact. Appealing to the proof of \cite[Theorem
VI.2.15]{diestel_uhl}, we may construct a sequence
$(f_n)_{n=1}^\infty$ of disjointly supported functions in $C(K)$ of
norm at most one such that $\inf_{n\in \mathbb{N}}\|Tf_n\|>0$. The
family $F=\big\{\sum_{k=1}^n f_k\big\}_{n=1}^\infty$ is then a
uniformly bounded infinite $\prec$-chain for which
$\inf\{\|Tf-Tg\|\colon f,g\in F, f\neq g\}>0$.\end{proof}

Given a compact Hausdorff space $K$ and a Banach space $X$, every
operator $T\colon C(K)\to X$ admits a~Riesz-type representation (\emph{cf}.~\cite[Chapter 6]{diestel_uhl}) in the following precise
sense: there exists a~$w^\ast$-countably additive vector measure
$\mu\colon\Sigma\to X^{\ast\ast}$ (called the {\it representing
measure} for $T$) on the $\sigma$-algebra $\Sigma$ of all Borel
subsets of $K$ such that:
\begin{romanenumerate}
\item for each $x^\ast\in X^\ast$ the map $\Sigma\ni A\mapsto \mu(A)x^\ast$ is a~regular countably additive scalar measure (and will be denoted  $x^\ast\circ\mu$);
\item the map $x^\ast\mapsto x^\ast\circ\mu$ from $X^\ast$ into $C(K)^\ast$ is $w^\ast$-to-$w^\ast$ continuous;
\item $x^\ast T(f)=\int_K f\,\mathrm{d}(x^\ast\circ\mu)$ for each $x^\ast\in X^\ast$ and $f\in C(K)$;
\item $\n{T}=\n{\mu}(K)$.
\end{romanenumerate}

The representing measure $\mu$ may be expressed explicitly by the formula $\mu(A)=T^{\ast\ast}\p_A$, where $\p_A\in C(K)^{\ast\ast}$ acts as $\p_A(\nu)=\nu(A)$ ($A\in \Sigma$, $\nu\in C(K)^\ast$). Equivalently, it may be defined by the prescription $\mu(A)x^\ast=\mu_{x^\ast}(A)$, where $\mu_{x^\ast}=T^\ast x^\ast$ is the scalar measure produced by the Riesz theorem applied to the functional $x^\ast T$.

We will need the following topological counterpart of Rosenthal's
lemma (\emph{cf}.~\cite[Lemma I.4.1]{diestel_uhl}). Although its
proof is almost the same as the original one, we present it for
completeness and to demonstrate the r\^{o}le played by extremal disconnectedness.

\begin{lemma}\label{rosenthal}
Let $K$ be a compact Hausdorff space which is extremally disconnected, 
  and let $(V_n)_{n=1}^\infty$ be a~sequence of pairwise
disjoint open subsets of $K$. Suppose that $(\mu_n)_{n=1}^\infty$ is
a~sequence of scalar Borel measures on $K$ having 
uniformly bounded variations. Then, for every $\e>0$ there exists
a~strictly increasing sequence $(n_k)_{k=1}^\infty$ of natural
numbers such that, for each $k\in \N$,
$$\abs{\mu_{n_k}}\Bigg(\oo{\bigcup_{j\neq k}V_{n_j}}\Bigg)<\e.$$
\end{lemma}
\begin{proof}
Let us suppose, with no loss of generality, that
$\abs{\mu_n}(K)\leqslant 1$ for each $n\in\N$. Consider any sequence
$(M_p)_{p=1}^\infty$ of pairwise disjoint infinite subsets of $\N$
such that $\N=\ccup_pM_p$. We consider two cases.

{\it Case 1.} First, if there is some $p\in\N$ for which
$$\abs{\mu_k}\Bigg(\oo{\bigcup_{\substack{j\in M_p\\ j\neq
k}}V_j}\Bigg)<\e\quad\mbox{for each }k\in M_p,$$then we get the
assertion from the induced enumeration of  $M_p$: $\{n_1<n_2< \cdots\}$. 

{\it Case 2.} Now suppose that for every $p\in\N$ there is $k_p\in M_p$ such that
\begin{equation}\label{R1}
\abs{\mu_{k_p}}\Bigg(\oo{\bigcup_{\substack{j\in M_p\\ j\neq k_p}}V_j}\Bigg)\geqslant\e.
\end{equation}
Fix, for a moment, any $p\in\N$. Since $\ccup_q V_{k_q}$ is disjoint from the open set $\ccup_{j\in M_p, j\neq k_p}V_j$, so is its closure. Hence,
\begin{equation}\label{R2}
\bigcup_{j\in M_p, j\neq k_p}V_j\subset\bigcup_{n\in\N} V_n\setminus\oo{\bigcup_{q\in\N} V_{k_q}}.
\end{equation}
Observe that $\oo{\ccup_qV_{k_q}}$ is open, as $K$ is extremally disconnected, whence $\oo{\ccup_nV_n}\setminus\oo{\ccup_qV_{k_q}}$ is closed. Therefore, by \eqref{R2}, we get
\begin{equation}\label{R3}
\oo{\bigcup_{\substack{j\in M_p\\ j\neq k_p}}V_j}\subset\oo{\bigcup_{n\in\N} V_n}\setminus\oo{\bigcup_{q\in\N}V_{k_q}}.
\end{equation}
Obviously, we have $$\abs{\mu_{k_p}}\Bigg(\oo{\bigcup_{q\in\N}V_{k_q}}\Bigg)+\abs{\mu_{k_p}}\Bigg(\oo{\bigcup_{n\in\N} V_n}\setminus\oo{\bigcup_{q\in\N}V_{k_q}}\Bigg)\leqslant 1,$$
so \eqref{R3} and \eqref{R1} imply $$\abs{\mu_{k_p}}\Bigg(\oo{\bigcup_{q\in\N}V_{k_q}}\Bigg)\leqslant 1-\e,$$and this inequality is valid for every $p\in\N$.

Consequently, we may repeat the same argument replacing the
space $K$ with the clopen subspace $\oo{\bigcup_{q\in\N}V_{k_q}}$
(which is extremally disconnected as this property is inherited by
open subspaces) and the sequences $(\mu_n)_{n=1}^\infty$ and
$(V_n)_{n=1}^\infty$ with $(\mu_{k_p})_{p=1}^\infty$ and
$(V_{k_p})_{p=1}^\infty$, respectively. By continuing this we would
get subsequent upper bounds $1-2\e, 1-3\e,\ldots$  for some of the
variations $\abs{\mu_n}$. Since this process has to terminate, we
will end up with {\it Case 1}, where the assertion has been proved.
\end{proof}

\begin{proposition}\label{bishopwc}
Let $K$ be a compact Hausdorff space which is extremally disconnected, 
  and let $X$ be a~Banach space.
Every operator from $C(K)$ into $X$ which satisfies {\rm
($\symbishop$)} is weakly compact.
\end{proposition}
\begin{proof}
It suffices to consider the real space $C_\R(K)$
since, we prove (arguing by contraposition) that there is a~$\delta$-$\prec$-chain in $C_\R(K)$ which will also do the job in $C_\C(K)$.

Assume that $T\colon C_\R(K)\to X$  is a~non-weakly compact
operator, and let $\mu$ be its representing measure. Then (see the
proof of Theorem 5.5.3 in \cite{albiac_kalton}) there exist a~number
$\e>0$, a~sequence $(O_n)_{n=1}^\infty$ of pairwise disjoint open
subsets of $K$ and a~sequence $(x_n^\ast)_{n=1}^\infty$ in the unit
ball of $X^\ast$ such that $(x_n^\ast\circ\mu)(O_n)>\e$ for each
$n\in\N$ (recall that $x_n^\ast\circ\mu=T^\ast x_n^\ast$). Since
$x_n^\ast\circ\mu$ is a~regular measure and $K$ is extremally
disconnected, we may assume that each $O_n$ is clopen. By Lemma
\ref{rosenthal}, we may also assume that
$$
\abs{x_n^\ast\circ\mu}\Bigg(\oo{\bigcup_{j\neq n}O_j}\Bigg)<\frac{\e}{2}\quad\mbox{for each }n\in\N .
$$

Let $\{q_n\colon n\in\N\}$ be an enumeration of the rational numbers and define
$A_t=\{n\colon q_n<t\}$ for $t\in\R$. Then $\{A_t\colon t\in\R\}$ is an increasing chain of infinite subsets of~$\N$. For each $t$, the set $C_t=\oo{\bigcup_{n\in A_t}O_t}$ is clopen, hence its
characteristic function, call it $f_t$, is continuous. The family $\{C_t\colon t\in\R\}$ is a strictly increasing chain of clopen sets, so that $s<t$ readily implies $f_s\prec f_t$. Clearly $\n{f_t}=1$ for all~$t$.

Now, let $s<t$. Then, there is $n\in A_t\setminus A_s$ (in fact the
set is infinite). Since $$\supp(f_t-f_s)\subset\oo{\bigcup_{k\in\N}O_k}=O_n\cup\oo{\bigcup_{k\not=n}O_k}$$ and $f_t-f_s$ equals $1$ on $O_n$, we get
\begin{equation*}
\begin{split}
\n{T(f_t)-T(f_s)} &\geqslant
     x_n^\ast T(f_t-f_s)=
     \int_K(f_t-f_s)\, \mathrm{d}(x_n^\ast\circ\mu)\\
    &= \Big(\int_{O_n}+\int_{\,\oo{\bigcup_{k\neq n}O_k}}\Big)(f_t-f_s)\, \mathrm{d}(x_n^\ast\circ\mu)\\
&\geqslant(x_n^\ast\circ\mu)(O_n)-\abs{x_n^\ast\circ\mu}\Bigg(\oo{\bigcup_{k\neq n}O_k}\Bigg)\\
&>\e-\frac{\e}{2}=\frac{\e}{2}\quad (i,j\in \mathbb{N}, i\neq j)
\end{split}
\end{equation*}
which proves that $T$ does not satisfy ($\symbishop$).

\end{proof}

We now introduce a counterpart of condition ($\symbishop$) for
vector measures. Namely, for a~set algebra $\Sigma$, Banach space $X$,
 and~(finitely) additive vector measure $\mu\colon\Sigma\to X$,
 consider the following property:
\begin{itemize*}
 \item[($\symrook$)]
for each uncountable chain (with respect to inclusion)
$\{E_i\}_{i\in I}$ in $\Sigma$,
$$\inf\{\n{\mu(E_i)-\mu(E_j)}\colon i,j\in I, i\not=j\}=0.$$
\end{itemize*}

A~vector measure $\mu\colon\Sigma\to X$ is called {\it strongly
additive} provided that for every sequence
$(E_n)_{n=1}^\infty\subset\Sigma$ of pairwise disjoint sets, the
series $\sum_{n=1}^\infty\mu(E_n)$ is unconditionally convergent in
$X$.

\begin{proposition}\label{measure_1}
Let $\Sigma$ be a $\sigma$-algebra of sets and $X$ be a Banach space. A~bounded vector measure $\mu\colon\Sigma\to X$ satisfies {\rm ($\symrook$)} if and only if it is strongly additive.
\end{proposition}
\begin{proof}
($\Leftarrow$) Suppose that $\mu$ does not satisfy ($\symrook$). Then there exists a~monotone (with respect to inclusion) sequence $(E_n)_{n=1}^\infty\subset\Sigma$ such that $\n{\mu(E_m)-\mu(E_n)}>\delta$ for $m,n\in\N$, $m\neq n$, and for some $\delta>0$. Put $D_n=E_{n+1}\setminus E_n$ provided $(E_n)_{n=1}^\infty$ is increasing and $D_n=E_n\setminus E_{n+1}$ otherwise. In any case $(D_n)_{n=1}^\infty$ forms a~sequence of pairwise disjoint elements from $\Sigma$ such that $\n{\mu(D_n)}>\delta$ for each $n\in\N$. Consequently, $\mu$ is not strongly additive.

\vspace*{2mm} ($\Rightarrow$) Now, suppose that $\mu$ fails to be
strongly additive. Then the Diestel--Faires theorem
 (\emph{cf}.~\cite[Theorem I.4.2]{diestel_uhl}) produces a closed subspace $Y$ of
$X$ and an isomorphism $T\colon \ell_\infty\to Y$ such that
$T(e_n)=\mu(A_n)$ for some pairwise disjoint sets
$(A_n)_{n=1}^\infty\subset\Sigma$. For any $n\in\N$ we have
$1=\n{e_n}\leqslant\n{T^{-1}}\cdot\n{\mu(A_n)}$, hence for some
$x_n^\ast$ in the unit ball of $X^\ast$ we have
$x_n^\ast\mu(A_n)\geqslant\n{T^{-1}}^{-1}$. Since
$(x_n^\ast\mu)_{n=1}^\infty$ is a~uniformly bounded sequence of
scalar measures, Rosenthal's lemma produces a~subsequence
$(x_{n_k}^\ast\mu)_{k=1}^\infty$ such that
$$\abs{x_{n_k}^\ast\mu}\Bigl(\bigcup_{j\neq
k}A_{n_j}\Bigr)<\frac{1}{2}\n{T^{-1}}^{-1}\quad\mbox{for each
}k\in\N.$$

Let $\Ca$ be an uncountable chain of subsets of $\N$ and for each $C\in\Ca$ define $E(C)=\ccup_{j\in C}A_{n_j}$. Plainly, $\{E(C)\}_{C\in\Ca}$ is an uncountable chain of members of $\Sigma$. Moreover, for any $C_1, C_2\in\Ca$ with $C_1\subsetneq C_2$, and any $k\in C_2\setminus C_1$, we have
\begin{equation}\label{RM3}
\begin{split}
\n{\mu(E(C_2))&-\mu(E(C_1))}=\n{\mu(E(C_2\setminus C_1))}=\Bigl\|\mu\Bigl(\bigcup_{j\in C_2\setminus C_1}A_{n_j}\Bigr)\Bigr\|\\
&\geqslant\Bigl|x_{n_k}^\ast\mu\Bigl(\bigcup_{j\in C_2\setminus C_1}A_{n_j}\Bigr)\Bigr|\geqslant \abs{x_{n_k}^\ast\mu(A_{n_k})}-\abs{x_{n_k}^\ast\mu}\Bigl(\bigcup_{j\neq k}A_{n_j}\Bigr)\\
&>\n{T^{-1}}^{-1}-\frac{1}{2}\n{T^{-1}}^{-1}=\frac{1}{2}\n{T^{-1}}^{-1}.
\end{split}
\end{equation}
This shows that $\mu$ fails to satisfy ($\symrook$), so the proof is complete.
\end{proof}

We are now prepared to proceed to the main result of this section.

\begin{theorem}\label{bishop_rook}
Let $K$ be a compact Hausdorff space which is extremally disconnected,
 and $X$ a~Banach space, let
 $\Sigma$ be the Borel $\sigma$-algebra of $K$ and
 let $T\colon C(K)\to X$ be an~operator with representing measure $\mu\colon\Sigma\to X^{\ast\ast}$.
 Then the following assertions are equivalent:
\begin{romanenumerate}
\item\label{weaklycompact} $T$ is weakly compact;
\item\label{repmeasure} $\mu$ is strongly additive;
\item\label{bishop} $T$ satisfies {\rm ($\symbishop$)};
\item\label{rook} $\mu$ satisfies {\rm ($\symrook$)}.
\end{romanenumerate}
\end{theorem}
\begin{proof}The equivalence of clauses (\ref{weaklycompact}) and (\ref{repmeasure}) is the Bartle--Dunford--Schwartz theorem
 (\emph{cf}.~\cite[Theorem VI.2.5]{diestel_uhl}) and is valid for any compact Hausdorff space $K$. The equivalence of clauses (\ref{weaklycompact}) and (\ref{bishop}) follows from Propositions \ref{wcompactisbishop} and \ref{bishopwc}. The equivalence of (\ref{repmeasure}) and (\ref{rook}) is immediate form Proposition \ref{measure_1}.
\end{proof}

It is easily seen that the implication ($\Leftarrow$) of Proposition \ref{measure_1} holds true for any set algebra $\Sigma$, not necessarily a~$\sigma$-algebra. However, the situation is not so clear for the implication ($\Rightarrow$), so the natural question arises: how can one characterise the class of set algebras $\Sigma$ for which every vector measure $\mu\colon\Sigma\to X$ satisfying ($\symrook$) is strongly additive? An~inspection of the proof of Proposition \ref{measure_1} suggests that some kind of interpolation property of $\Sigma$ could do the job, so we may make our problem more precise. A~set algebra $\Sigma$ has the {\it subsequential completeness property} (SCP) whenever for every sequence $(E_n)_{n=1}^\infty$ of pairwise disjoint sets from $\Sigma$ there is a~subsequence $(E_{n_k})_{k=1}^\infty$ with $\ccup_kE_{n_k}\in\Sigma$. (Haydon constructed a~set algebra with (SCP) which is not a~$\sigma$-algebra; see \cite[Proposition 1E]{haydon}.)
\begin{question}\label{SCP}
Suppose that a set algebra $\Sigma$ has (SCP). Is it true that, for any Banach space $X$, every vector measure
$\mu\colon\Sigma\to X$ satisfying ($\symrook$) is necessarily
strongly additive?
\end{question}

\begin{remark}
The second and third-named authors (\cite{kaniakochanek}) have
recently studied the operator ideal $\mathscr{W\!C\!G}$ of
\emph{weakly compactly generated} operators, that is, operators
whose range is contained in a~weakly compactly generated subspace of
their codomain. (A Banach space is \emph{weakly compactly generated}
whenever it contains a linearly dense weakly compact subset.) The
class $\mathscr{W\!C\!G}$ contains all weakly compact operators and
all operators having separable range, but in contrast, a weakly
compactly generated operator defined on a~$C(K)$-space need not
satisfy ${\rm (\symbishop)}$ (\emph{cf}.~Proposition
\ref{wcompactisbishop}).

To see this, consider the ordinal interval $K=[0,\omega_1]$ equipped with the order topology which makes it a~compact Hausdorff space. Let $D = \{0\}\cup\{\alpha+1\colon\alpha<\omega_1\}$. Define a~mapping $\varphi\colon [0,\omega_1]\to [0,\omega_1]$ by $$\varphi(\alpha)=\left\{\begin{array}{ll}
\alpha+1, & \mbox{if }\alpha\in D,\\
\alpha, & \mbox{if }\alpha\in [0,\omega_1]\setminus D.
\end{array}\right.$$
Plainly, $\varphi$ is continuous, hence the operator $C_\p\colon C[0,\omega_1]\to C[0,\omega_1]$ defined by $C_\p f=f\circ\p$ is bounded.

We note that $T=I_{C[0,\om_1]}-C_\p$ maps the Schauder basis $\{\ind_{[0,\alpha]}\}_{0\leqslant\alpha\leqslant\om_1}$ of $C[0,\om_1]$ onto the set $\{\mathbf{1}_{\{\alpha\}}\}_{\alpha\in D}\cup\{0\}$. Consequently, the range of $T$ is isomorphic to $c_0(\omega_1)$, which is a weakly compactly generated Banach space, so that $T\in \mathscr{W\!C\!G}(C[0,\omega_1])$. On the other hand, $T$ maps the $1$-$\prec$-chain $\{\ind_{[0,\alpha]}\}_{\alpha\in D}$ onto $\{\ind_{\{\alpha\}}\}_{\alpha\in D}$, so $T$ fails ${\rm (\symbishop)}$.\end{remark}

It is natural to ask whether the condition  ${\rm (\symbishop)}$ is the same as a seemingly weaker similar statement {\rm($\symbishop_{\rm{wo}}$)} in which the $\delta$-$\prec$-chains are assumed to be well-ordered. This is, however, not the case, as we shall see using the following observation.
\begin{proposition}\label{wobishop}Let $K$ be a compact Hausdorff space satisfying c.c.c.~Then the identity map $I_{C(K)}$ satisfies {\rm($\symbishop_{\rm{wo}}$)}.
\end{proposition}
\begin{proof} Let $K$ be a compact Hausdorff space which does not satisfy the condition {\rm($\symbishop_{\rm{wo}}$)}. With no loss of generality, we may assume that, for some $\delta>0$, there exists an uncountable, well-ordered $\delta$-$\prec$-chain $\{f_\alpha\}_{\alpha<\omega_1}\subset C(K)$. For each ordinal $\alpha<\omega_1$ set $g_\alpha = |f_{\alpha+1}-f_\alpha|$. The family $\{g_\alpha^{-1}\big((\delta, \infty)\big)\colon \alpha<\omega_1\}$ consists of uncountably many pairwise disjoint open sets, so $K$ fails the c.c.c.~condition. \end{proof}

\begin{corollary}The condition {\rm($\symbishop_{\rm{wo}}$)} is strictly weaker than {\rm($\symbishop$)}.\end{corollary}
\begin{proof}In the light of Proposition \ref{wobishop}, it is sufficient to exhibit an example of a compact Hausdorff space which is c.c.c., yet fails {\rm($\symbishop$)}.

Let $K$ be the \emph{double arrow space}, that is, $K$ is the set $((0,1]\times \{0\})\cup ([0,1)\times \{1\})$ endowed with the order topology arising from the lexicographic order. It is well known that $K$ is a first-countable, non-metrisable, compact Hausdorff space which is separable (hence also c.c.c.), see \cite[3.10.C]{Eng}. We note that for each $0<b<1$ the order-interval $I_b:=\big[(0,1), (b,0)\big]$ is clopen in $K$, which means that its characteristic function $f_b = \mathbf{1}_{I_b}$ is continuous. Consequently, $\{f_b\}_{0<b<1}$ is an uncountable $1$-$\prec$-chain, so $K$ fails $(\symbishop)$, as desired.\end{proof}

\section{Compact Hausdorff spaces enjoying $(\symbishop)$}
In this section we analyse certain compact spaces to show that the identity operator on a $C(K)$-space can satisfy ($\symbishop$). We note that there is no loss of generality in considering only
$\prec$-chains consisting of non-negative functions. Indeed, if
$f\prec g$ then $|f|\prec |g|$. In the remainder of this section, we
shall therefore assume that every $\prec$-chain has this property.

For a $\prec$-chain $F$, we denote by $A(F,f)$ the set of
predecessors of $f$, that is, $A(F,f)=\{g\in F\colon g\prec f\}$;
furthermore we set
 $$Z(F,f)=\{x\in K\colon
f(x)\neq0\}\setminus\bigcup\{\supp(g)\colon g\in A(F,f)\}.$$

\begin{lemma}\label{Zi}
Let $K$ be a compact Hausdorff space, let $\delta>0$ and let
$F\subset C(K)$ be a $\delta$-$\prec$-chain. Then, for each $f\in
F$, $Z(F,f)$ is nonempty; in particular there is $x\in Z(F,f)$
such that $f(x)\geqslant\delta$.
\end{lemma}

\begin{proof}
This is clear if $f$ has a direct predecessor $g$ in $F$; in that case
$$Z(F,f)=\{x\in K\colon x\notin\supp(g)\mbox{ and }f(x)>0\}.$$
As $\|f-g\|\geqslant\delta$ there must be $x\in K\setminus\supp(g)$ with $f(x)\geqslant\delta$;
hence $x$ belongs to~$Z(F,f)$.

In case where $f$ does not have a direct predecessor, we take an increasing and
cofinal (possibly transfinite) sequence $\{g_\alpha\colon \alpha<\theta\}$
in $A(F,f)$.
For each $\alpha$ we pick
$x_\alpha\in\supp(g_{\alpha+1})\setminus\supp(g_\alpha)$
such that $g_{\alpha+1}(x_\alpha)\geqslant\delta$.
Note that this implies that $f(x_\alpha)\geqslant\delta$ for all~$\alpha$;
therefore $f(x)\geqslant\delta$ for all $x\in L$, where
$$L=\bigcap_{\alpha<\theta}\overline{\{x_\beta\colon \beta\geqslant\alpha\}}.$$
On the other hand, if $g\prec f$, then $g\prec g_\alpha$ for some~$\alpha$
and then $g(x_\beta)=0$ for $\beta\geqslant\alpha$ and hence $g(x)=0$
for $x\in L$.
This yields that the intersection $L\cap\supp(g)$ is empty for all $g\in A(F,f)$ and hence that $L\subseteq Z(F,f)$.
\end{proof}

\begin{corollary}
Every compact Hausdorff space which is countable satisfies {\rm $(\symbishop)$}.
\end{corollary}

For an order type $\alpha$ we denote by $\alpha^\ast$ the
reverse of $\alpha$.

\begin{proposition}\label{thm:omegaD}
Let $K$ be the one-point compactification of a discrete
space~$\Gamma$, let $\delta>0$, and let a
$\delta$-$\prec$-chain $F\subset C(K)$ be given. Then there are
$\alpha,\beta\leqslant\omega$ such that the order type of~$F$ is
$\alpha+\beta^\ast$
\end{proposition}

\begin{proof}
Let $F$ be a $\delta$-$\prec$-chain. Split $F$ into $F_1=\{f\colon f(\infty)=0\}$ and $F_2=\{f\colon f(\infty)>0\}$. (Recall that we are assuming that $F$ consists of non-negative functions only.)
Note that $f\prec g$ whenever $f\in F_1$ and $g\in F_2$.

For $f\in F_1$ set $\Upsilon_f=\{x\colon f(x)\geqslant \delta\}$; this set is finite.
If $f\prec g$ in $F_1$ then $\Upsilon_f$ is a proper subset of~$\Upsilon_g$.
This implies that the order type of $F_1$ is at most $\omega$.

Likewise, for $f\in F_2$ set $\Upsilon_f=\{x\colon f(x)=0\}$, which is again a finite set.
In this case $\Upsilon_f$ contains properly $\Upsilon_g$ whenever $f\prec g$ in $F_2$.
It follows that $F_2$ is order-isomorphic to a subset of $\{-n\colon n\in\omega\}\subseteq \mathbb{Z}$.
\end{proof}

We have now arrived at one of the main results of this section. We
shall prove that local connectedness is a sufficient condition for
absence of uncountable $\delta$-$\prec$-chains of functions on
spaces satisfying c.c.c. A topological space is \emph{locally
connected} if each point has a neighbourhood basis consisting
of connected sets. The disjoint union of finitely many copies of the
unit interval is an easy example of a (linearly ordered) compact
space which is locally connected, but not connected.
\begin{theorem}\label{cellular}
Let $K$ be a compact Hausdorff space which is locally connected and
 let $\delta>0$. The cardinality of any~$\delta$-$\prec$-chain in $K$ does not exceed the cellularity of $K$.\end{theorem}

\begin{proof}
 Let $F=\{f_i\}_{i\in I}\subset C(K)$ be a $\delta$-$\prec$-chain
for some $\delta>0$. The sets $Z(F,f_i)$ ($i\in I$) are non-empty
and pairwise disjoint. It is enough to prove that each $Z(F,f_i)$
($i\in I$) is open, as this will immediately yield the inequality
$|I|\leqslant c(K)$.

Fix $i\in I$ and $x\in Z(F,f_i)$; our aim is to prove that $x$ lies in the interior of $Z(F,f_i)$. Choose an open connected neighbourhood $U\subseteq K$ of $x$ such that $f(y)>\tfrac{1}{2}f(x)$ for each $y\in U$. We \emph{claim} that $U\cap \supp(f_j)=\varnothing$ for each $f_j\prec f_i$, which means that $U\subseteq Z(F,f_i)$. Suppose that this is not the case, that is, $U\cap \supp(f_j)\neq\varnothing$ for some $f_j\prec f_i$. Because $\supp(f_j)$ is the closure of the open set $V=\{w\in K\colon f_j(w)\neq 0\}$, there must be $y\in U\cap \supp(f_j)$ such that $f_j(y)\neq 0$. Since $U$ is connected, $x\in U\setminus \supp(f_j)$ and $y\in U\cap V$, the set $U$ intersects the boundary of $V$; let $z$ be an element of this intersection. Then, $z\in U$, so $f_i(z)\neq 0$. On the other hand, $z\in\supp(f_j)$, yet $z$ does not belong to $V$, so $0=f_j(z)=f_i(z)$, as $f_j\prec f_i$; a contradiction.\end{proof}
\begin{corollary}\label{souslin}
Every locally connected compact Hausdorff space which satisfies
 c.c.c.~also satisfies {\rm $(\symbishop)$}.\end{corollary}
\begin{proof}
 This is a direct reformulation of Theorem \ref{cellular} for spaces
having countable cellularity.\end{proof}

 \begin{remark}\label{lcexamples}
 (i)
 Compact spaces which are connected need not be locally connected;
a~standard example of such a~space (which is also not
path-connected) is the so-called \emph{topologist's closed sine
curve}, that is, the graph of $\sin(1/x),\,x\in (0,1]$ with the
interval $\{0\}\times [-1,1]$ adjoined, endowed with the relative
Euclidean topology. Nevertheless, it is easily seen from the very
definition of the order topology that every connected linearly
ordered compact space is also locally connected.

Evidently, the unit interval satisfies the assumptions of Corollary
\ref{souslin}. The existence of other linearly ordered
connected (hence locally connected) examples of compact spaces
satisfying c.c.c.~is equivalent to the negation of the Souslin
hypothesis, $\mathsf{SH}$, which is consistent with and independent
of $\mathsf{ZFC}$. Namely, under $\neg\mathsf{SH}$, one may consider
the two-point compactification of a Souslin line (recall that
a~\emph{Souslin line} is a~non-separable, linearly ordered,
connected space without end-points which satisfies c.c.c.).
Consistently, there may exist non-homeomorphic Souslin lines.

 (ii)
 There are consistent examples of non-metrisable, locally connected
compact spaces which are hereditarily Lindel\"{o}f and hereditarily
separable (hence also c.c.c.) (see \cite{filipov} and \cite{kunen}).
Curiously, there may be no non-trivial convergent sequences in
spaces satisfying the assumptions of Corollary \ref{souslin} as
shown by van Mill under $\mathsf{CH}$ (\cite{vanmill}).
Assuming Jensen's diamond principle $\diamondsuit$, one has an
example of such a space which is moreover one-dimensional
 (\emph{cf}.~\cite{hakunen}).


We refer to \cite[Chapter b-11]{ency} for an exposition concerning local connectedness and further examples.
 \end{remark}

Corollary \ref{souslin} is optimal in the sense that there may exist uncountable $\prec$-chains, yet for any $\delta>0$ there may be no uncountable $\delta$-$\prec$-chains.
\begin{example}\label{cantor}
Let $\Delta\subseteq [0,1]$ be the ternary Cantor set. For each $d\in \Delta$ set $f_d\colon [0,1]\to [0,1]$ by $f_d(x) = \mbox{dist}(x, \Delta)\cdot\mathbf{1}_{[d,1]}(x)$. Each function $f_d$ ($d\in \Delta$) is continuous and the family $\{f_d\}_{d\in \Delta}$ forms an uncountable $\prec$-chain in $C[0,1]$. On the other hand, it is not hard to see that it is not a $\delta$-$\prec$-chain for any positive $\delta$.\end{example}

Now, we exhibit another example of a~compact, totally disconnected Hausdorff space which enjoys ($\symbishop$). This is a~well-known construction in point-set topology, where it is sometimes called a~{\it ladder system space}.

We equip the ordinal number $\omega_1$ with a~topology as follows:
\begin{itemize}
\item we declare zero and each countable successor number to be an isolated point;
\item for each non-zero limit ordinal $\lambda$ we choose a set $\{\alpha_{n,\lambda}\colon n<\omega\} \subseteq \lambda$ of order type $\omega$ consisting of successor numbers and cofinal in $\lambda$ (a \emph{ladder}); then we define basic open neighbourhoods of $\lambda$ to be of the form $U_{\lambda, m}=\{\lambda\}\cup\{\alpha_{n,\lambda}\colon n\geqslant m\}$ ($m<\omega$).
\end{itemize}
Of course, the topology on $\omega_1$ depends on the choice of
ladders but in any case the space $\omega_1$ topologised in this
manner is first countable, locally compact and Hausdorff.

\begin{theorem}\label{ladder}
Let $K$ be the one-point compactification of a ladder system space on~$\omega_1$, and let
$\delta>0$. Then every $\delta$-$\prec$-chain in $C(K)$ is countable. In particular, $K$ satisfies {\rm ($\symbishop$)}.
\end{theorem}

\begin{proof}
We model the proof after that of Proposition~\ref{thm:omegaD}:
we take a $\delta$-$\prec$-chain $F$ and divide it into two sets:
$F_1=\{f\in F\colon f(\omega_1)=0\}$ and $F_2=\{f\in F\colon f(\omega_1)>0\}$.

Let $L_\lambda$ denote the ladder associated to $\lambda$,  and set $L_\lambda^+=L_\lambda\cup\{\lambda\}$.
Note that a closed subset of~$\omega_1$ is compact if and only if it is covered by
finitely many sets of the form $L_\lambda^+$ and possibly a finite number of isolated points.

We \emph{claim} that $F_1$ is countable. For $f$ in $F_1$ let $\Upsilon_f=\{x\in\omega_1\colon f(x)\geqslant \delta\}$.
Each~$\Upsilon_f$ is compact, hence countable. Lemma \ref{Zi} implies that if $g\prec f$, then $Z(F,g)\cap \Upsilon_f\neq  \varnothing$. As the sets $Z(F,f)$ and $Z(F,g)$
are disjoint, this shows that $A(F,f)$ is countable for any $f\in F_1$.
Also, if $g\prec f$ in $F_1$, then $\Upsilon_g$ is a proper subset of~$\Upsilon_f$.

It suffices to show that $F_1$ has countable cofinality.
We assume, seeking of a contradiction, that
$\{ f_\alpha\colon \alpha<\omega_1\}$ is a strictly increasing sequence in $F_1$.
Then the sets $\Upsilon_{f_\alpha}$ form a strictly increasing sequence as well
so that the union $\bigcup_{\alpha<\omega_1}\Upsilon_{f_\alpha}$ is uncountable.

Given $\alpha<\omega_1$ let $\beta_\alpha>\alpha$ be such that
\begin{itemize}
\item $\bigcup_{\gamma\leqslant\alpha}\Upsilon_{f_\gamma}\subseteq \beta_\alpha$, and
\item there is a $\gamma<\beta_\alpha$ such that $x>\alpha$ for some
      $x\in \Upsilon_{f_\gamma}$.
\end{itemize}
Next let $C$ be a closed and unbounded set such that for each $\gamma\in C$ we have $\beta_\alpha<\gamma$ whenever $\alpha<\gamma$.

Let $\gamma\in C$ be such that $C\cap\gamma$ has order type~$\omega^2$.
We show that $\Upsilon_{f_\gamma}$ is not compact.
To this end, we note that, by the above two items there is a point in $\Upsilon_{f_\gamma}$
between $\eta$ and~$\xi$, whenever $\eta$ and~$\xi$ are consecutive elements of $C\cap\gamma$.
This shows that the order type of $\Upsilon_{f_\gamma}\cap\gamma$ is at least~$\omega^2$.

Let $H$ be a finite subset of~$\Lambda$.
We show that the set $\Upsilon_{f_\gamma}\setminus\bigcup_{\lambda\in H}L_\lambda^+$ is infinite.
This follows from the following three observations
\begin{itemize}
\item if $\lambda\in H\cap\gamma$, then $L_\lambda^+$ is bounded below by $\gamma$;
\item $\Upsilon_{f_\gamma}\setminus L_\gamma^+$ is cofinal in~$\gamma$
      (because of the order types of both sets); and
\item if $\lambda\in H$ is larger than~$\gamma$ then $L_\lambda^+\cap\gamma$
      is finite.
\end{itemize}
Thus we see that $\Upsilon_{f_\gamma}$ is not compact.

\smallskip
For $f\in F_2$ we consider the compact set $\Upsilon_f=\{x\in X\colon f(x)=0\}$.
Arguing similarly as before we conclude that $F_2$ is countable, yet the order of $F_2$ is reversed in the following sense: if $f\prec g$ in $F_2$,
then $\Upsilon_f\supseteq \Upsilon_g$ and there is $x\in \Upsilon_f\setminus \Upsilon_g$ such that $g(x)\geqslant\delta$.\end{proof}

A compact Hausdorff space is \emph{Eberlein} if it is homeomorphic
to a weakly compact subset of a Banach space. Countable compact
Hausdorff spaces, the one-point compactification of a discrete space
(Proposition \ref{thm:omegaD}) as well as the unit interval (\emph{cf}.~Remark \ref{lcexamples}~(i) 
 are classical examples of Eberlein compact spaces. A space $K$ is
Eberlein if and only if the Banach space $C(K)$ is weakly compactly
generated (\emph{cf}.~\cite[Theorem 2]{AL}). Every weakly
compactly generated Banach space is necessarily Lindel\"of in its
weak topology. Wage observed that no ladder system space $K$ is
Eberlein (\cite{wage}), while Pol proved that the Banach
space $C(K)$ is Lindel\"of in its weak topology (\cite{pol}).
In the light of these observations, we raise the following
question.
\begin{question}\label{eberleinbishop}
Does the class of compact Hausdorff spaces satisfying ($\symbishop$) contain all Eberlein compact
spaces? More generally, does every compact Hausdorff space $K$ for which the Banach space $C(K)$ is Lindel\"{o}f in its weak topology satisfy ($\symbishop$)?\end{question}

\section{The left ideal structure of operators satisfying ($\symbishop$)}
Let $[A]^2$ stand for the set of all $2$-element subsets of a~given
set $A$. The Dushnik and Miller partition lemma
 (\cite{dushnik_miller}) asserts that, for every infinite regular aleph
$\aleph_\alpha$, and any colouring $c\colon
[\aleph_\alpha]^2\to\{0,1\}$, at least one of the following
conditions holds true:
\begin{itemize*}
\item[(i)] there is a~set $A\subset\aleph_\alpha$ with $\abs{A}=\aleph_\alpha$ and $[A]^2\subset c^{-1}\{0\}$;
\item[(ii)] for every set $A\subset\aleph_\alpha$ with $\abs{A}=\aleph_\alpha$ there are $a\in A$ and a~set $B\subset A$ such that $\abs{B}=\aleph_\alpha$ and $\{a,b\}\in c^{-1}\{1\}$ for each $b\in B$.
\end{itemize*}

\begin{proposition}\label{dushnik_miller}
Let $K$ be a compact Hausdorff space and let $X$ be a~Banach space. If operators $T,U\colon C(K)\to X$ satisfy {\rm ($\symbishop$)}, then so does $T+U$.
\end{proposition}
\begin{proof}
Assume, in search of a contradiction, that there are $\delta>0$ and a~$\prec$-chain $\{f_i\}_{i\in I}$ in $C(K)$, where $I$ has cardinality $\aleph_1$, such that $\sup_{i\in I}\n{f_i}<\infty$ and \begin{equation*}\label{sum}\tag{*}\n{(T+U)(f_i)-(T+U)(f_j)}\geqslant\delta\quad(i,j\in I,\, i\not=j).\end{equation*} Set $$\mathcal{I}=\bigl\{\{i,j\}\in [I]^2:\,\n{U(f_i)-U(f_j)}\leqslant\delta/4\bigr\}$$
and consider the partition $[I]^2=\mathcal{I}\cup ([I]^2\setminus\mathcal{I})$.

If assertion (i) from the partition lemma holds (where $\alpha=1$ and $c^{-1}\{0\}=[I]^2\setminus\mathcal{I}$), then we get a~contradiction with the fact that $U$ satisfies ($\symbishop$).

If assertion (ii) is valid (where $c^{-1}\{1\}=\mathcal{I}$), then there is an $i_0\in I$ and a~set $B\subset I$ with $\abs{B}=\aleph_1$, such that $\{i_0,j\}\in\mathcal{I}$ for each $j\in B$. Then for every $j,k\in B$ we have $$\n{U(f_j)-U(f_k)}\leqslant\n{U(f_{i_0})-U(f_j)}+\n{U(f_{i_0})-U(f_k)}\leqslant\delta/2,$$hence our assumption (\ref{sum}) implies that $\n{T(f_j)-T(f_k)}\geqslant\delta/2$ for all $j,k\in B$, $j\not=k$, which contradicts the condition ($\symbishop$) for $T$.
\end{proof}

\begin{theorem}\label{left_ideal}
For every compact Hausdorff space $K$ the set of all operators $T\in \mathscr{B}(C(K))$ satisfying {\rm ($\symbishop$)} forms a~closed left ideal of the Banach algebra $\mathscr{B}(C(K))$.
\end{theorem}
\begin{proof}
From the very definition of ($\symbishop$) it is evident that if $(T_n)_{n=1}^\infty\subset\mathscr{B}(C(K))$ norm converges to some $T\in\mathscr{B}(C(K))$ and each $T_n$ satisfies ($\symbishop$), then $T$ does as well.

Now, suppose that $T\in\mathscr{B}(C(K))$ satisfies ($\symbishop$) and let $S\in\mathscr{B}(C(K))$. If for some $\delta>0$ there were an~uncountable chain $\{f_i\}_{i\in I}$ with $\n{ST(f_i)-ST(f_j)}>\delta$ for all $i,j\in I$, $i\neq j$, then for all such $i,j$ we would also have $\n{T(f_i)-T(f_j)}>\delta/\n{S}$; a~contradiction. Finally, an~appeal to Proposition \ref{dushnik_miller} completes the proof.
\end{proof}

We have already noticed that for $K$ extremally disconnected the family of operators $T\colon C(K)\to C(K)$ which satisfy ($\symbishop$) coincides with
the two-sided ideal of weakly compact operators. Examples in Section 3 demonstrate that in some particular cases every operator on a $C(K)$-space may satisfy ($\symbishop$), which means that the left ideal of operators satisfying ($\symbishop$) is again a~two-sided ideal, the improper ideal
$\mathscr{B}(C(K))$. This raises the following open question.
\begin{question}
Is the set of operators on a $C(K)$-space which satisfy ($\symbishop$) always a right, and hence a two-sided ideal of $\mathscr{B}(C(K))$?\end{question}

\medskip
\noindent \textbf{Acknowledgements.}
The authors would like to express their gratitude to N.J. Laustsen  and J.M. Lindsay for careful reading of preliminary versions of this paper.

\bibliographystyle{amsplain}

\end{document}